\documentclass[oneside, a4paper]{amsart}

\makeatletter
\let\@wraptoccontribs\wraptoccontribs
\makeatother

\usepackage{config}
\usepackage{makros}
\usepackage{makros-letters}

\title{On the Right Derived Functors of Ordinary Parts}

\date{}

\author{Manuel Hoff}
\address{Manuel Hoff: Fakultät für Mathematik, Universität Bielefeld, D-33501 Bielefeld}
\email{manuel.hoff@uni-bielefeld.de}

\author{Sarah Diana Meier}
\address{Sarah Diana Meier: Fakultät für Mathematik, Universität Bielefeld, D-33501 Bielefeld}
\email{smeier@math.uni-bielefeld.de}

\author{Michael Spieß}
\address{Michael Spieß: Fakultät für Mathematik, Universität Bielefeld, D-33501 Bielefeld}
\email{mspiess@math.uni-bielefeld.de}

\contrib[with an appendix joint with]{Claudius Heyer}
\address{Claudius Heyer: Fakultät für Mathematik, Universität Paderborn, D-33098 Paderborn}
\email{heyer@math.upb.de}

\begin{document}

\begin{abstract}
    We prove a variant of Emerton's conjecture concerning the right derived functors of the ordinary parts functor $\Ord_P^G$.
    This functor plays an important role in the theory of mod $p$ representations of $p$-adic reductive groups.
    A key ingredient for our proof is a comparison between certain small and parabolic inductions.
    Additionally, our method yields an explicit description of Vignéras' right adjoint to parabolic induction. 
    
    In the appendix (joint with Heyer) we apply our results to obtain a mod $p$ variant of Bernstein's Second Adjointness, i.e.\ we show that the right and left adjoint of derived parabolic induction are isomorphic (on complexes with admissible cohomology) up to a cohomological shift and twist by a character. 
\end{abstract}

\maketitle

\tableofcontents

\section{Introduction} \label{sec-intro}

Let\extrafootertext{This research was funded by the Deutsche Forschungsgemeinschaft (DFG, German Research Foundation)
– Project-ID 491392403 – TRR 358} $G$ be a $p$-adic reductive group, that is the group of $F$-points of a connected reductive algebraic group over a $p$-adic field $F$.
Let $P\subseteq G$ be a parabolic subgroup with Levi decomposition $P = M N$ and opposite parabolic $\ovrP=M\ovrN$.
For a commutative Noetherian ring $R$ we consider the functor of parabolic induction
\[
	\Ind_{\ovrP}^G \colon \Mod_R^{\sm}(M) \longrightarrow \Mod_R^{\sm}(G) \, ,
\]
where $\Mod_R^{\sm}(M)$ and $\Mod_R^{\sm}(G)$ denote the categories of smooth representations of $M$ and $G$ over $R$ respectively.
It is an exact functor, right adjoint to the functor assigning to a smooth $G$-representation $V$ its $\ovrN$-coinvariants $V_{\ovrN}$.

In \parencite{em1} Emerton has introduced a left exact functor -- the functor of ordinary parts -- in the opposite direction
\[
	\Ord_P^G \colon \Mod_R^{\sm}(G) \longrightarrow \Mod_R^{\sm}(M) \, .
\]
Its significance lies in the fact that, when working within the subcategories of admissible representations, it becomes right adjoint to parabolic induction.
More precisely, both $\Ind_{\ovrP}^G$ and $\Ord_P^G$ preserve admissibility and their restrictions to the categories $\Mod_R^{\adm}(M)$ and $\Mod_R^{\adm}(G)$ of admissible smooth representations naturally form an adjoint pair, see \parencite{em1} and \parencite{vig}.
The functor of ordinary parts plays an important role in the study of mod $p$ and $p$-adic representations of $p$-adic reductive groups and, in particular, in the $p$-adic Langlands program.
It can also be used to provide a representation theoretic framework of Hida's theory of nearly ordinary cohomology. 

In the sequel \parencite{em2} Emerton addresses the question of an explicit description of the right derived functors of $\Ord_P^G$.    
For that he considers the sequence of functors $\HH^n \Ord_P^G \colon \Mod_R^{\sm}(G) \to \Mod_R^{\sm}(M)$ (with $\HH^0 \Ord_P^G = \Ord_P^G$) defined by
\[
	\HH^n \Ord_P^G V
	\, \coloneqq \,
	\roundbr[\Big]{\Ind_{M^+}^M \HH^n(N_0, V)}^{\lf{A_M}} \, .
\]
Here $N_0 \subseteq N$ denotes a fixed compact open subgroup, $M^+ = \set{m \in M}{m N_0 m^{-1} \subseteq N_0} \subseteq M$ is the submonoid of positive elements and $A_M$ is the maximal split central subtorus of $M$.
The suffix \enquote*{$\lf{A_M}$} indicates taking $A_M$-locally finite vectors, i.e.\ passing to the subrepresentation consisting of those $w$ such that $R[A_M] \cdot w$ is finitely generated over $R$.

In the case where $R$ is an Artinian local ring with finite residue field of characteristic $p$, it is shown in loc.\ cit.\ that $\HH^{\bullet} \Ord_P^G$ defines a $\delta$-functor when restricted to $\Mod_R^{\adm}(G)\subseteq \Mod_R^{\sm}(G)$ or more generally to the subcategory $\Mod_R^{\ladm}(G) \subseteq \Mod_R^{\sm}(G)$ of locally admissible smooth representations of $G$ over $R$.
It is expected that in fact $\HH^{\bullet} \Ord_P^G \colon \Mod_R^{\ladm}(G)\to \Mod_R^{\ladm}(M)$ is a universal $\delta$-functor, see \parencite[Conjecture 3.7.2]{em2}. 
Our first main result is related to this conjecture.

\begin{theoremletter}[\ref{thm-right-derived-ord}] \label{theo:emertonconj}
	Let $R$ be an Artinian local ring with residue field of characteristic $p$, and let $V$ be a locally admissible $G$-representation over $R$.
	Then we have
	\[
		\bbR^n \Ord_P^G V
		\, \cong \,
		\HH^n \Ord_P^G V
		\, \cong \,
        \ind_{M^+}^M \HH^n(N_0, V)
	\]
	for all $n \geq 0$.
\end{theoremletter}
 
We remark that, unlike in \parencite{em2}, we consider here the derived functors of $\Ord_P^G$ on the larger category $\Mod_R^{\sm}(G)$ instead of $\Mod_R^{\ladm}(G)$.
In the case when $G = \PGL_2(F)$ this result has been established in \parencite{spi}.
The original conjecture has been proven for $G = \GL_2(F)$ in \parencite{empa}. 

A crucial ingredient in the proof of \cref{theo:emertonconj} is a comparison theorem between certain small and parabolically induced $G$-representations which we think is of independent interest.
To state this result we consider a compact open subgroup $K_0 \subseteq G$ that admits an Iwahori decomposition $K_0 = N_0 M_0 \ovrN_0$ for compact open subgroups $M_0 \subseteq M$ and $\ovrN_0 \subseteq \ovrN$.
Let $W$ be a smooth representation of $K_0$ over $R$ that is trivial on $N_0$ and $\ovrN_0$.
Then the induced $G$-representation $\ind_{K_0}^G W$ carries a natural right Hecke $A_M^+$-action and we have the following result.

\begin{theoremletter}[\ref{thm-comparison}]
\label{theo:comparison}
	There exists a canonical isomorphism
	\[
		\ind_{A_M^+}^{A_M} \ind_{K_0}^G W
		\, \cong \,
		\Ind_{\ovrP}^G \ind_{M_0}^M W
	\]
	of smooth $G$-representations.
\end{theoremletter}

Our strategy of proof, carried out in \cref{sec-localized-compact-induction}, is to construct a $G$-equivariant sheaf on the profinite set $X = G/\ovrP$ whose global sections is the left hand side of the isomorphism and whose stalk at $x_0 \coloneqq \ovrP \in X$ is $\ind_{M_0}^M W$.
The earliest similar result has been established in \parencite{bali} for $\GL_2(F)$.
More recent statements of this kind have been made in \parencite{hen-vig} (for special parahoric subgroups $K_0$) and also \parencite{herz1} (for split reductive groups).

We also briefly describe the strategy of proof of \cref{theo:emertonconj} given in \cref{sec-derived-functors}.
We choose a suitable decreasing sequence $(K_i)_{i \geq 0}$ of compact open subgroups of $G$ such that each $K_i$ admits an Iwahori decomposition and such that $\bigcap_i K_i = N_0$.
We then have
\[
    \bbR^n \Ord_P^G V
	\, \cong \,
	\colim_i \bbR^n \Ord_{P, K_i}^G V
\]
as smooth $A_M$-representations, where $\Ord_{P, K_i}^G V \coloneqq (\Ind_{A_M^+}^{A_M} V^{K_i})^{\lf{A_M}}$.
For the derived functors of $\Ord_{P, K_i}^G$ we show that for $V$ a locally admissible $G$-representation we have 
\[
	\bbR^n \Ord_{P, K_i}^G V
	\, \cong \,
	\roundbr[\Big]{\Ind_{A_M^+}^{A_M} \HH^n(K_i, V)}^{\lf{A_M}} \, .
\]
To prove this, we crucially use that $\ind_{A_M^+}^{A_M} \ind_{K_i}^G \trivrep$ is flat over $R[A_M/A_M \cap K_i]$, a fact that follows from \cref{theo:comparison}.

It turns out that \cref{theo:comparison} can also be used to give a rather explicit description of the right adjoint $\R_{\ovrP}^G$ of parabolic induction on the category of all smooth representations (i.e.\ without restricting to admissible or locally admissible ones) that has been studied in \parencite{vig}.
We show in \cref{cor-describing-right-adjoint} below that for a smooth $G$-representation $V$ we have
\[
	\R_{\ovrP}^G V \, \cong \, \colim_i \Ind_{A_M^+}^{A_M} V^{K_i} \, .
\]
Moreover, if $V$ is admissible and $R$ is Artinian local with residue field of characteristic $p$, then for the right derived functors of $\R_{\ovrP}^G$ we get
\[
	\bbR^n \R_{\ovrP}^G V
	\, \cong \,
	\bbR^n \Ord_P^G V
	\, \cong \,
	\HH^n \Ord_P^G V \, ,
\]
see \cref{cor-ord-rpg}. 
In special cases for $G = \SL_2(\bbQ_p)$ this has been proven in \parencite[]{koziol}.

\medskip

In the appendix (joint with Claudius Heyer) we prove a mod $p$ variant of Bernstein's Second Adjointness for complex coefficients (cf. \parencite[]{bernstein}). 

\begin{theoremletter}[\ref{thm-sec-adj}]
	Let $V^{\bullet}$ be a complex of smooth $G$-representations with admissible cohomology.
	Then the natural morphism
 	\[
 		\omega \otimes_R^{\bbL} \bbR \R_{\ovrP}^G V^{\bullet} 
 		\longrightarrow \bbL_P^G V^{\bullet}
 	\]
 	is an equivalence. 
\end{theoremletter}

Here $\bbL_P^G$ is the left adjoint (introduced in \cite{heyer-left-adjoint}) to the derived parabolic induction functor $\Ind_P^G$ and $\omega$ is a certain smooth character of $M$ placed in cohomological degree $-\dim_{\bbQ_p} N$. 

In \parencite{datetal}, Bernstein's result has been generalized to an arbitrary coefficient ring $R$ in which $p$ is invertible.
While this result holds for general smooth representations, we cannot expect the same in our setting as the right adjoint $\R_{\ovrP}^G$ does not commute with arbitrary direct sums, see \parencite[Proposition 4.24]{abe}.

\medskip

\subsection*{Acknowledgments}

We like to thank Florian Herzig, Simon Paege, Vytautas Paškūnas and Paul Schneider for very helpful discussions and comments. Special thanks go to Claudius Heyer who suggested the application to Second Adjointness and also contributed to some of our results (such as \cref{cor-right-derived-rpg} and \cref{cor-ord-rpg}). 

\medskip

\section{Setup and Notation} \label{sec-setup}

\subsection*{Smooth representations} \label{subsec-smooth-reps}

Fix a commutative coefficient ring $R$.
Let $H$ be a locally profinite monoid; by this we mean a Hausdorff topological monoid that admits a profinite open subgroup $U \subseteq H$ such that $H/U$ and $U \backslash H$ are discrete.

We write $\Mod_R^{\sm}(H)$ for the Grothendieck abelian category of \emph{smooth ($R$-linear) $H$-representations}, i.e.\ of those $R[H]$-modules $V$ such that every vector $v \in V$ is fixed by an open subgroup $U \subseteq H$.
When $H$ is discrete we simply write $\Mod_R(H)$ instead of $\Mod_R^{\sm}(H)$.

Given a locally profinite closed submonoid $I \subseteq H$ we write
\[
    \Ind_I^H \colon \Mod_R^{\sm}(I) \longrightarrow \Mod_R^{\sm}(H)
\]
for the \emph{smooth induction functor}, the right adjoint to restriction $\Res_I^H \colon \Mod_R^{\sm}(H) \to \Mod_R^{\sm}(I)$. It is concretely given by sending a smooth $I$-representation $W$ to $\Ind_I^H W \coloneqq \Hom_I(R[H], W)^{\sm}$.
Given a subset $H' \subseteq H$ that is stable under left multiplication by $I$ and right multiplication by some open subgroup of $H$ we also write $\Ind_I^{H'} W \coloneqq \Hom_I(R[H'], W)^{\sm}$.
Given two such subsets $H'' \subseteq H' \subseteq H$ we have a natural restriction map $\Ind_I^{H'} W \to \Ind_I^{H''} W$.
When $H$ is a locally profinite group and $I \subseteq H$ a closed subgroup then we also write $\cInd_I^H \subseteq \Ind_I^H$ for the \emph{compact induction functor} and use a similar notation for $H' \subseteq H$ as above.

Similarly, given an open submonoid $I \subseteq H$ we write
\[
    \ind_I^H \colon \Mod_R^{\sm}(I) \longrightarrow \Mod_R^{\sm}(H)
\]
for the \emph{small induction functor}, the left adjoint to restriction. It is concretely given by sending a smooth $I$-representation $W$ to $\ind_I^H W \coloneqq R[H] \otimes_{R[I]} W$.
For an element $w \in W$ we write $1_{I, w} \coloneqq 1 \otimes w \in \ind_I^H W$.
Given a subset $H' \subseteq H$ that is stable under right multiplication by $I$ we also write $\ind_I^{H'} W \coloneqq R[H'] \otimes_{R[I]} W$.
Given two such subsets $H'' \subseteq H' \subseteq H$ we have a natural map $\ind_I^{H''} W \to \ind_I^{H'} W$ and when $I$ is a locally profinite group, then this map even admits a canonical retraction coming from the decomposition $\ind_I^{H'} W = \ind_I^{H''} W \oplus \ind_I^{H' \setminus H''} W$.

\medskip

Suppose that $R$ is Noetherian.
We say that a smooth $H$-representation $V$ is \emph{finite} if it is finitely generated as an $R$-module and denote by $\Mod_R^{\sm, \f}(H) \subseteq \Mod_R^{\sm}(H)$ the full subcategory spanned by the finite representations.
This subcategory contains $0$ and is stable under taking subobjects, quotients and extensions.

Given a locally profinite closed central submonoid $H' \subseteq H$ such that $H'/U$ is finitely generated for every (or equivalently one) compact open subgroup $U \subseteq H'$, we say that a smooth $H$-representation $V$ is \emph{$H'$-locally finite} if $R[H'] \cdot v$ is a finitely generated $R$-module for all $v \in V$ and denote by $\Mod_R^{\sm, \lf{H'}}(H)$ the full subcategory of $\Mod_R^{\sm}(H)$ spanned by the  $H'$-locally finite $H$-representations.
This subcategory contains $0$ and is stable under taking subobjects, quotients, extensions and arbitrary direct sums.
We write
\[
    (\blank)^{\lf{H'}} \colon \Mod_R^{\sm}(H) \longrightarrow \Mod_R^{\sm, \lf{H'}}(H)
\]
for the right adjoint to the inclusion.
In the special case $H' = H$ we simply write $\Mod_R^{\sm, \lf{}}(H)$ instead of $\Mod_R^{\sm, \lf{H}}(H)$.

We say that a smooth $H$-representation $V$ is \emph{admissible} if for every open subgroup $U \subseteq H$ the invariants $V^U$ form a finitely generated $R$-module.
We say that $V$ is \emph{locally admissible} if $R[H] \cdot v \subseteq V$ is admissible for all $v \in V$.

\medskip

\subsection*{Homological Algebra} \label{subsec-hom-alg}

Given a Grothendieck abelian category $\cA$ we write $\cK(\cA)$ and $\cD(\cA)$ for the (unbounded) homotopy category and the derived category of cochain complexes in $\cA$ respectively.
For a complex $V^{\bullet} \in \cD(\cA)$ we write $\HH^n(V^{\bullet}) \in \cA$ for its $n$-th cohomology object.
For a locally profinite monoid $H$ we also use the shortened notation $\cD(H) \coloneqq \cD(\Mod_R^{\sm}(H))$. 

Given a left exact functor $F \colon \cA \to \cB$ between Grothendieck abelian categories we write
\[
    \bbR F \colon \cD(\cA) \longrightarrow \cD(\cB)
\]
for its total right derived functor and $\bbR^n F \colon \cD(\cA) \to \cB$ for its $n$-th derived functor.

\medskip

\subsection*{Algebraic Groups} \label{subsec-alg-grps}

We fix a prime number $p$ and a $p$-adic local field $F$.
We also fix a connected reductive group $\bfG$ over $F$, a parabolic $\bfP \subseteq \bfG$ with Levi decomposition $\bfP = \bfM \bfN$ and opposite parabolic $\bfovrP = \bfM \bfovrN$, and write $\bfA_{\bfM}$ for the maximal split central torus of $\bfM$.
We denote by $G$, $P$, $M$, $N$, $\ovrP$, $\ovrN$, $A_M$ the locally profinite groups of $F$-rational points of $\bfG$, $\bfP$, $\bfM$, $\bfN$, $\bfovrP$, $\bfovrN$, $\bfA_{\bfM}$.

Furthermore, we fix a descending sequence $(K_i)_{i \geq 0}$ of torsionfree compact open subgroups $K_i \subseteq G$ and an open submonoid $A_M^+ \subseteq A_M$ with the following properties.
\begin{enumerate}
    \item 
    The compact open subgroup $N_0 \coloneqq K_i \cap N \subseteq N$ is independent of the choice of $i \geq 0$.

    \item
    Each $K_i$ is Iwahori-decomposed with respect to $(N, M, \ovrN)$, i.e.\ we have $K_i = N_0 M_i \ovrN_i$ for compact open subgroups $M_i \subseteq M$ and $\ovrN_i \subseteq \ovrN$.
    We write $P_i \coloneqq M_i N_i$ and $\ovrP_i \coloneqq M_i \ovrN_i$.

    \item
    We have $\bigcap_i K_i = N_0$, or equivalently $\bigcap_i M_i = 1$ and $\bigcap_i \ovrN_i = 1$.

    \item
    For every $a \in A_M^+$ we have $a N_0 a^{-1} \subseteq N_0$ and $a^{-1} \ovrN_i a \subseteq \ovrN_i$.

    \item
    There exists $b_0 \in A_M^+$ such that $A_M=A_M^+[b_0^{-1}]$ (such $b_0$ is automatically strongly positive in the sense of \parencite[Definition 6.16]{buko}).
\end{enumerate}
Such data always exist, see \parencite[Lemma 3.4.1]{em2} and \parencite[Lemma 6.14]{buko} for the last property.
Furthermore, according to the following lemma, which is a slight variant of \parencite[Lemma 3.2.1]{em1}, we may also assume that for all compact open subgroups $U \subseteq A_M^+$ the quotient $A_M^+/U$ is finitely generated.

\begin{lemma}
    Let $A$ be a commutative locally profinite group such that for every or equivalently one compact open subgroup $U \subseteq A$ the quotient $A/U$ is finitely generated.
    Furthermore, let $A^+ \subseteq A$ be an open submonoid such that there exists $b_0 \in A^+$ such that $A = A^+[b_0^{-1}]$.

    Then there exists another open submonoid $\tilA^+ \subseteq A$ contained in $A^+$, containing $b_0$ and still satisfying $A = \tilA^+[b_0^{-1}]$, and such that for every or equivalently one compact open subgroup $U \subseteq \tilA^+$ the quotient $\tilA^+/U$ is finitely generated.
\end{lemma}

\begin{proof}
    After replacing $A$ by $A/U$ and $A^+$ by $A^+/U$ for a compact open subgroup $U \subseteq A^+$, we may assume that $A$ is discrete.
    Now choose a finite system of generators $(e_1, \dotsc, e_r)$ of $A$ as a monoid and a non-negative integer $m$ such that $b_0^m e_i \in A^+$ for all $i = 1, \dotsc, r$.
    Then
    \[
        \tilA^+ \, \coloneqq \, \anglebr[\big]{\set{b_0^m e_i}{i = 1, \dotsc, r} \cup \curlybr{b_0}} \, \subseteq \, A^+
    \]
    satisfies the required conditions.
\end{proof}

We also write 
\[
    M^+ \, \coloneqq \, \set[\big]{m \in M}{m N_0 m^{-1} \subseteq N_0}
\]
and $P^+ \coloneqq N_0 M^+$. Note that $M^+ \subseteq M$ and $P^+ \subseteq P$ both are open submonoids and that we have $A_M^+ \subseteq M^+ \cap A_M$ and $M = M^+[b_0^{-1}]$.

\medskip
\section{Comparing Localized Small Induction with Parabolic Induction} \label{sec-localized-compact-induction}

Fix a compact open subgroup $K \subseteq G$ that admits an Iwahori decomposition
with respect to $(N, M, \ovrN)$ and such that we have $a (K \cap N) a^{-1} \subseteq K \cap N$ and $a^{-1} (K \cap \ovrN) a \subseteq K \cap \ovrN$ for all $a \in A_M^+$.
Also fix a smooth $K$-representation $(W, \tau)$ that is trivial on $K \cap N$ and $K \cap \ovrN$.

\medskip

The goal of this section is to prove that there is a natural isomorphism
\[
    \ind_{A_M^+}^{A_M} \ind_K^G W \, \cong \, \Ind_{\ovrP}^G \ind_{K \cap M}^M W \, ,
\]
see \Cref{thm-comparison} below.
In order to make sense of the left hand side, we need to introduce a certain Hecke action of $A_M^+$ on $\ind_K^G W$, see the following lemma.

\begin{lemma} \label{lem-hecke-actions}
    There are natural right actions of $A_M^+$ on the smooth $G$-representation $\ind_K^G W$ and on the smooth $\ovrP$-representation $\ind_{K \cap \ovrP}^{\ovrP} W$ as well as a natural right action of $A_M$ on the smooth $M$-representation $\ind_{K \cap M}^M W$ that are given by
    \begin{gather*}
        1_{K, w} \bullet a \, \coloneqq \, \sum_{n \in (K \cap N)/a (K \cap N) a^{-1}} n a \cdot 1_{K, w} \, , \\
        1_{K \cap \ovrP, w} \bullet a \, \coloneqq \, a \cdot 1_{K \cap \ovrP, w}
        \qquad \text{and} \qquad
        1_{K \cap M, w} \bullet a \, \coloneqq \, a \cdot 1_{K \cap M, w}
    \end{gather*}
    for $a \in A_M^+$ (respectively $a \in A_M$) and $w \in W$ and that are trivial on $\ker(\tau) \cap A_M$.
    Moreover, the natural $\ovrP$-equivariant restriction maps
    \[
        \ind_K^G W \longrightarrow \ind_{K \cap \ovrP}^{\ovrP} W \longrightarrow \ind_{K \cap M}^M W
    \]
    are compatible with these actions.
\end{lemma}

\begin{proof}
    Let us show that given an element $a \in A_M^+$ the map
    \[
        W \longrightarrow \ind_K^G W \, , \qquad w \mapsto \sum_{n \in (K \cap N)/a (K \cap N) a^{-1}} n a \cdot 1_{K, w}
    \]
    is $K$-equivariant; this implies that we obtain a $G$-equivariant map $(\blank) \bullet a \colon \ind_K^G W \to \ind_K^G W$.
    For this we note that the natural map $(K \cap N)/a (K \cap N) a^{-1} \to K/(K \cap a K a^{-1})$ is a bijection.
    For $k \in K$ and $w \in W$ we thus obtain
    \begin{align*}
        k \cdot \sum_{n \in (K \cap N)/a (K \cap N) a^{-1}} na \cdot 1_{K, w}
        & \, = \, \sum_{k' \in K/(K \cap a K a^{-1})} k k' a \cdot 1_{K, \tau(k'^{-1}) w}
        \\
        & \, = \, \sum_{k'' \in K/(K \cap a K a^{-1})} k'' a \cdot 1_{K, \tau(k''^{-1}) \tau(k) w}
    \end{align*}
    as desired. Note that in this calculation we used that $\tau$ is trivial on $K \cap N$ and $K \cap \ovrN$.
    Showing that the other two actions are well-defined is done similarly.

    Now let us also show that the map $\alpha \colon \ind_K^G W \to \ind_{K \cap \ovrP}^{\ovrP} W$ projecting onto the identity summand of the $\ovrP$-equivariant Mackey decomposition (see \parencite[Theorem 1.1]{yamamoto})
    \[
        \ind_K^G W \, \cong \, \bigoplus_{g \in \ovrP \backslash G/K} \ind_{g K g^{-1} \cap \ovrP}^{\ovrP} g_* W
    \]
    is compatible with the action of $A_M^+$ on both sides. Note that this map is explicitly given by
    \[
        g \cdot 1_{K, w} \longmapsto
        \begin{cases}
            \ovrp \cdot 1_{K \cap \ovrP, \tau(k) w} & \text{if $g = \ovrp k \in \ovrP K$}
            \\
            0 & \text{if $g \notin \ovrP K$}
        \end{cases} \, .
    \]
    For $a \in A_M^+$, $g \in G$ and $w \in W$ we have
    \[
        \alpha \roundbr[\big]{g \cdot 1_{K, w} \bullet a} \, = \, \sum_{n \in (K \cap N)/a (K \cap N) a^{-1}} \alpha \roundbr[\big]{g n a \cdot 1_{K, w}} \, .
    \]
    Now $g n a \in \ovrP K = \ovrP (K \cap N)$ if and only if we have $g \in \ovrP (a (K \cap N) a^{-1}) n^{-1}$ so that all but possibly one of the terms in the sum above vanish.
    If $g \notin \ovrP (K \cap N)$, we obtain
    \[
        \alpha \roundbr[\big]{g \cdot 1_{K, w} \bullet a} \, = \, 0 \, = \, \alpha(g \cdot 1_{K, w}) \bullet a \, ,
    \]
    and if $g = \ovrp n \in \ovrP (K \cap N)$, we obtain
    \[
        \alpha \roundbr[\big]{g \cdot 1_{K, w} \bullet a} \, = \, \alpha(\ovrp a \cdot 1_{K, w}) \, = \, \ovrp a \cdot 1_{K \cap \ovrP, w} \, = \, \alpha(g \cdot 1_{K, w}) \bullet a \, . \qedhere
    \]
\end{proof}

\medskip

Consider the profinite set $X \coloneqq G/\ovrP$ and write $x_0 \coloneqq \ovrP \in X$ for its base point and $\pi \colon G \to X$ for the projection.
In the following by a \emph{(pre-)sheaf of $R$-modules on $X$} we always mean a (pre-)sheaf on the basis of compact open subsets of $X$.

We recall that a \emph{smooth $G$-equivariant structure} on a (pre-)sheaf $\cH$ of $R$-modules on $X$ is a collection of isomorphisms
\[
    g \cdot (\blank) \colon \cH \longrightarrow g_* \cH \, , \qquad \roundbr[\big]{s \in \cH(U)} \mapsto \roundbr[\big]{g \cdot s \in \cH(g U)}
\]
for $g \in G$ such that we have
\[
    1 \cdot s \, = \, s
    \quad \text{and} \quad
    (g g') \cdot s \, = \, g \cdot (g' \cdot s)
\]
for all $g, g' \in G$ and $s \in \cH(U)$, and such that every section $s \in \cH(U)$ is fixed by a compact open subgroup of the stabilizer $\Stab_G(U) = \set{g \in G}{g U = U} \subseteq G$.  

The basic idea for the proof of \cref{thm-comparison} now is to use the following general result.

\begin{prop} \label{prop-classify-equivariant-sheaves}
    The functor
    \[
        \curlybr[\Bigg]{
        \begin{gathered}
            \text{smooth $G$-equivariant sheaves} \\
            \text{of $R$-modules on $X$}
        \end{gathered}
        }
        \longrightarrow \Mod_R^{\sm} \roundbr[\big]{\ovrP}
    \]
    sending a $G$-equivariant sheaf $\cH$ to its stalk $\cH_{x_0}$ is an equivalence of categories.
    Moreover, for a smooth $G$-equivariant sheaf $\cH$ on $X$ we have
    \[
        \Gamma(X, \cH) \, \cong \, \Ind_{\ovrP}^G \cH_{x_0} \, .
    \]
\end{prop}

\begin{proof}
    Given a smooth $\ovrP$-representation $V$, we can define a smooth $G$-equivariant sheaf $\cH^{(V)}$ on $X$ by setting
    \[
        \cH^{(V)}(U) \, \coloneqq \, \Ind_{\ovrP}^{\pi^{-1}(U)^{-1}} V
    \]
    for a compact open subset $U \subseteq X$.
    This construction gives a functor $V \mapsto \cH^{(V)}$ that is an inverse to $\cH \mapsto \cH_{x_0}$.
    The claim about the global sections then follows because we clearly have $\cH^{(V)}(X) = \Ind_{\ovrP}^G V$.
    See also \parencite[Chapitre III.3.2]{ren} for more details.
\end{proof}

\medskip

\subsection{The Presheaf $\cF$} \label{subsec-presheaf-f}

We define a presheaf $\cF$ of $R$-modules on $X$ by setting
\[
    \cF(U) \, \coloneqq \, \ind_K^{\pi^{-1}(U) K} W
\]
for a compact open subset $U \subseteq X$.
Given an element $g \in G$, we have $\pi^{-1}(g U) = g \pi^{-1}(U)$ so that the action of $g$ on $\ind_K^G W$ restricts to an isomorphism
\[
    g \cdot (\blank) \colon \ind_K^{\pi^{-1}(U) K} W \longrightarrow \ind_K^{\pi^{-1}(g U) K} W \, .
\]
In this way we obtain a $G$-equivariant structure on $\cF$.
Note that we have natural isomorphisms of smooth $\ovrP$-representations
\[
    \cF_{x_0} \, \cong \, \ind_K^{\ovrP K} W \, \cong \, \ind_{K \cap \ovrP}^{\ovrP} W \, .
\]
For convenience let us write $G_U \coloneqq \pi^{-1}(U) K$ and $G^U \coloneqq G \setminus G_U$.

\begin{lemma} \label{lem-hecke-submodule}
    Let $U \subseteq X$ be a compact open subset.
    Then $\ind_K^{G^U} W \subseteq \ind_K^G W$ is stable under the action of $A_M^+$.
\end{lemma}

\begin{proof}
    Note that we have $G_U = \pi^{-1}(U) (K \cap N)$ by the Iwahori decomposition.
    For $g \in \pi^{-1}(U)$, $n \in K \cap N$ and $a \in A_M^+$ we now have
    \[
        g n a^{-1} \, = \, (g a^{-1}) (a n a^{-1}) \, \in \, \pi^{-1}(U) (K \cap N) \, ,
    \]
    using that $\pi^{-1}(U)$ is stable under right multiplication by $\ovrP$.
    Thus we see that $G_U (A_M^+)^{-1} \subseteq G_U$ and consequently $G^U A_M^+ \subseteq G^U$; the claim of the lemma then follows.
\end{proof}

Using \Cref{lem-hecke-submodule} and the identification $\ind_K^{G_U} W \cong \ind_K^G W / \ind_K^{G^U} W$, we obtain a right action of $A_M^+$ on $\cF$ that commutes with the $G$-equivariant structure (and is trivial on $\ker(\tau) \cap A_M$).
The isomorphism $\cF_{x_0} \cong \ind_{K \cap \ovrP}^{\ovrP} W$ is compatible with the $A_M^+$-actions on both sides, see \cref{lem-hecke-actions}.

\medskip

\subsection{The sheaf $\ind_{A_M^+}^{A_M} \cF$} \label{subsec-sheaf-locf}

For $g \in G/K$ we define the compact open subset $U(g) \coloneqq g K \ovrP/\ovrP \subseteq X$.
Note that for a compact open subset $U \subseteq X$ we then have
\[
    G_U \, = \, \set[\big]{g \in G}{U(g) \cap U \neq \emptyset} \, .
\]

\begin{lemma} \label{lem-neighborhood-basis}
    For $r \geq 0$ we have
    \[
        U(1) \, = \, \bigsqcup_{n \in (K \cap N)/b_0^r (K \cap N) b_0^{-r}} U(n b_0^r)
    \]
    and these open covers form a cofinal system of open covers of $U(1) \subseteq X$.
\end{lemma}

\begin{proof}
    As $b_0$ is strictly positive, the compact open subgroups $b_0^r (K \cap N) b_0^{-r} \subseteq N$ for $r \geq 0$ form a basis of open neighborhoods of $1 \in N$.
    Now the claim follows as the natural map $N \to X$ is a homeomorphism onto an open subset of $X$ that maps $1$ to $x_0$ and $n b_0^r (K \cap N) b_0^{-r}$ onto $U(n b_0^r)$.
\end{proof}

\begin{lemma} \label{lem-f-loc-sheaf}
    The presheaf $\ind_{A_M^+}^{A_M} \cF \colon U \mapsto \ind_{A_M^+}^{A_M} \cF(U)$ is a sheaf.
\end{lemma}

\begin{proof}
    As $X$ is profinite, in order to verify the sheaf condition, it suffices to treat finite disjoint open covers.
    So, let $U = \bigsqcup_{i \in I} U_i$ be such a finite disjoint open cover of a compact open subset $U \subseteq X$.
    We have to show that
    \[
        \alpha \colon \ind_K^{G_U} W \longrightarrow \prod_{i \in I} \ind_K^{G_{U_i}} W
    \]
    becomes an isomorphism after applying $\ind_{A_M^+}^{A_M}$.
    As we have $G_U = \bigcup_i G_{U_i}$, we see that $\alpha$ is injective and thus stays injective after localization.

    In order to show that $\alpha$ becomes surjective after localization, we first claim that for a compact open subset $U' \subseteq X$ the $A_M$-representation $\ind_{A_M^+}^{A_M} \ind_K^{G_{U'}} W$ is generated by elements of the form $g \cdot 1_{K, w}$ where $U(g) \subseteq U'$.
    For this we note that for some general $g \in G_{U'}$ and $w \in W$ we have
    \[
        g \cdot 1_{K, w} \bullet b_0^r \, = \, \sum_n g n b_0^r \cdot 1_{K, w} \, \in \, \ind_K^{G_{U'}} W
    \]
    where the sum is over those $n \in (K \cap N)/b_0^r (K \cap N) b_0^{-r}$ such that $U(g n b_0^r) \cap U' \neq \emptyset$.
    It follows from \cref{lem-neighborhood-basis} that for $r \gg 0$ each $U(g n b_0^r)$ is either completely contained in $U'$ or disjoint to it.
    This proves the claim.

    Now given some $j \in I$, $g \in G$ with $U(g) \subseteq U_j$ and $w \in W$, $\alpha$ maps the element $g \cdot 1_{K, w} \in \ind_K^{G_U} W$ to $(\delta_{i j} g \cdot 1_{K, w})_i$.
    Together with the above claim this implies the desired surjectivity.
\end{proof}

\begin{lemma} \label{lem-computing-stalks}
    We have a natural isomorphism of smooth $\ovrP$-representations
    \[
        \roundbr[\big]{\ind_{A_M^+}^{A_M} \cF}_{x_0} \, \cong \, \ind_{K \cap M}^M W
    \]
    that is compatible with the action of $A_M$.
\end{lemma}

\begin{proof}
    We have already seen that $\cF_{x_0} \cong \ind_{K \cap \ovrP}^{\ovrP} W$.
    Thus we have to show that the map of smooth $\ovrP$-representations
    \[
        \ind_{K \cap \ovrP}^{\ovrP} W \longrightarrow \ind_{K \cap M}^M W \, , \qquad 1_{K \cap \ovrP, w} \mapsto 1_{K \cap M, w}
    \]
    becomes an isomorphism after applying $\ind_{A_M^+}^{A_M}$ to the left hand side.
    We start by noting that the map $\ind_{K \cap \ovrP}^{\ovrP} W \to \ind_{K \cap M}^M W$ has an $M$-equivariant section
    \[
        \ind_{K \cap M}^M W \longrightarrow \ind_{K \cap \ovrP}^{\ovrP} W \, , \qquad 1_{K \cap M, w} \mapsto 1_{K \cap \ovrP, w}
    \]
    that is again compatible with the action of $A_M^+$.
    We claim that this section becomes surjective after applying $\ind_{A_M^+}^{A_M}$.
    So, suppose we are given $\ovrp \in \ovrP$ and $w \in W$ and write $\ovrp = m \ovrn$ with $m \in M$ and $\ovrn \in \ovrN$.
    For $r \gg 0$ we then have
    \[
        \ovrp b_0^r \, = \, (m b_0^r)(b_0^{-r} \ovrn b_0^r) \, \in \, M (K \cap \ovrN)
    \]
    so that the element
    \[
        \ovrp \cdot 1_{K \cap \ovrP, w} \bullet b_0^r \, = \, \ovrp b_0^r \cdot 1_{K \cap \ovrP, w}
    \]
    is contained in the image of $\ind_{K \cap M}^M W \to \ind_{K \cap \ovrP}^{\ovrP} W$ as desired.
\end{proof}

\begin{theorem} \label{thm-comparison}
    There exists a natural isomorphism of smooth $G$-representations
    \[
        \ind_{A_M^+}^{A_M} \ind_K^G W \, \cong \, \Ind_{\ovrP}^G \ind_{K \cap M}^M W
    \]
    that is compatible with the action of $A_M$ on both sides.
\end{theorem}

\begin{proof}
    Apply \cref{prop-classify-equivariant-sheaves} to $\ind_{A_M^+}^{A_M} \cF$, using \cref{lem-f-loc-sheaf} and \cref{lem-computing-stalks}.
\end{proof}

\medskip
\section{\texorpdfstring{The Right Adjoint $\R_{\ovrP}^G$}{The Right Adjoint RPbarG}} \label{sec-right-adjoint}

As an application of \cref{thm-comparison} we now give a description of the right adjoint
\[
    \R_{\ovrP}^G \colon \Mod_R^{\sm}(G) \longrightarrow \Mod_R^{\sm}(M)
\]
to the parabolic induction functor $\Ind_{\ovrP}^G$, see \parencite[Proposition 4.2]{vig}.

\medskip

Given a smooth $G$-representation $(V, \sigma)$ we have an action of $M^+$ on $V^{N_0}$ that is given by
\[
    m \bullet v \, \coloneqq \, \sum_{n \in N_0/m N_0 m^{-1}} \sigma(n m) v
\]
for $m \in M^+$ and $v \in V^{N_0}$.
This action is smooth as every element $v \in V^{K_i}$ is fixed by the open subgroup $M_i \subseteq M^+$.
This yields a right action of $M^+$ on the pro-object $\formallim_i \ind_{K_i}^G \trivrep$ corepresenting the functor $\Mod_R^{\sm}(G) \to \Mod_R, \, V \mapsto V^{N_0} \cong \colim_i V^{K_i}$ that extends the $A_M^+$-action defined in \cref{lem-hecke-actions} and that is smooth in the sense that for every $i_0 \geq 0$ and $m \in M_{i_0}$ the diagram
\[
\begin{tikzcd}
    \displaystyle\formallim_i \ind_{K_i}^G \trivrep \ar[rr, "(\blank) \bullet m"] \ar[dr, swap, "\pr_{i_0}"]
    & & \displaystyle\formallim_i \ind_{K_i}^G \trivrep \ar[ld, "\pr_{i_0}"]
    \\
    & \ind_{K_{i_0}}^G \trivrep
\end{tikzcd}
\]
is commutative.
Similarly we also obtain a smooth right action of $M$ on the pro-object $\formallim_i \ind_{M_i}^M \trivrep$ corepresenting the forgetful functor $\Mod_R^{\sm}(M) \to \Mod_R$ that extends the $A_M$-action defined earlier.

\begin{lemma} \label{lem-comparison-m-equivariant}
    The isomorphism
    \[
        \formallim_i \ind_{A_M^+}^{A_M} \ind_{K_i}^G \trivrep \, \cong \, \formallim_i \Ind_{\ovrP}^G \ind_{M_i}^M \trivrep
    \]
    in $\Pro(\Mod_R^{\sm}(G))$ from \cref{thm-comparison} is compatible with the action of $M$ on both sides.
\end{lemma}

\begin{proof}
    Suppose we are given $m \in M^+$ and $i \geq 0$ and choose $j \geq 0$ such that $m \bullet 1_{K_i} \in (\ind_{K_i}^G \trivrep)^{K_j}$ and $m \bullet 1_{M_i} \in (\ind_{M_i}^G \trivrep)^{M_j}$.
    We have to show that the diagram
    \[
    \begin{tikzcd}[column sep = large]
        \ind_{K_j}^G \trivrep \ar[r, "(\blank) \bullet m"] \ar[d, "\alpha_j"]
        & \ind_{K_i}^G \trivrep \ar[d, "\alpha_i"]
        \\
        \ind_{M_j}^M \trivrep \ar[r, "(\blank) \bullet m"]
        & \ind_{M_i}^M \trivrep
    \end{tikzcd}
    \]
    is commutative.
    But this follows from the same argument as in the proof of \cref{lem-hecke-actions}.
\end{proof}

\begin{theorem} \label{cor-describing-right-adjoint}
    Let $V$ be a smooth $G$-representation.
    Then we have a natural isomorphism of smooth $M$-representations
    \[
        \R_{\ovrP}^G V \, \cong \, \colim_i \Ind_{A_M^+}^{A_M} V^{K_i} \, .
    \]

    Moreover, given a smooth $M$-representation $W$, the $M^+$-equivariant composition
    \[
        W \xlongrightarrow{\eta_W} \R_{\ovrP}^G \Ind_{\ovrP}^G W \longrightarrow \roundbr[\big]{\Ind_{\ovrP}^G W}^{N_0} \, ,
    \]
    where the first map is the unit of the adjunction between $\Ind_{\ovrP}^G$ and $\R_{\ovrP}^G$ and the second map comes from the description above, is concretely given by
    \[
        w \longmapsto
        \roundbr[\Bigg]{g \mapsto \begin{cases} m \cdot w & \text{if $g = \ovrn m n \in \ovrN M N_0 = \ovrP N_0$} \\ 0 & \text{if $g \notin \ovrP N_0$} \end{cases}} \, .
    \]
\end{theorem}

\begin{proof}
    Using \cref{thm-comparison}, we obtain $A_M$-equivariant isomorphisms
    \[ \label{equ-ast}\tag{$\ast$}
    \begin{aligned}
        \roundbr[\big]{\R_{\ovrP}^G V}^{M_i} & \, \cong \, \Hom_M \roundbr[\big]{\ind_{M_i}^M \trivrep, \R_{\ovrP}^G V}  \\
        & \, \cong \, \Hom_G \roundbr[\big]{\Ind_{\ovrP}^G \ind_{M_i}^M \trivrep, V} \\
        & \, \cong \, \Hom_G \roundbr[\big]{\ind_{A_M^+}^{A_M} \ind_{K_i}^G \trivrep, V} \\
        & \, \cong \, \Ind_{A_M^+}^{A_M} V^{K_i} \, ,
    \end{aligned}
    \]
    where in the last step we use Tensor-Hom adjunction. Consequently, this yields an $A_M$-equivariant isomorphism $\R_{\ovrP}^G V \cong \colim_i \Ind_{A_M^+}^{A_M} V^{K_i}$.
    \cref{lem-comparison-m-equivariant} then shows that this isomorphism is $M$-equivariant.

    Now let $W$ be a smooth $M$-representation.
    Then, for $i \geq 0$, we have a commutative diagram
    \[
    \begin{tikzcd}[column sep = large]
        W^{M_i} \ar[r, "\eta_W"] \ar[d, "{\cong}"]
        & \roundbr[\big]{\R_{\ovrP}^G \Ind_{\ovrP}^G W}^{M_i} \ar[d, "{\cong}"]
        \\
        \Hom_M \roundbr[\big]{\ind_{M_i}^M \trivrep, W} \ar[r, "\eta_{W, *}"] \ar[rd, "\Ind_{\ovrP}^G", swap]
        & \Hom_M \roundbr[\big]{\ind_{M_i}^M \trivrep, \R_{\ovrP}^G \Ind_{\ovrP}^G W} \ar[d, "{\cong}"]
        \\
        & \Hom_G \roundbr[\big]{\Ind_{\ovrP}^G \ind_{M_i}^M \trivrep, \Ind_{\ovrP}^G W}
    \end{tikzcd} \, .
    \]
    The second part of the theorem now follows from combining this with the calculation \eqref{equ-ast} and the observation that the $G$-equivariant map $\ind_{K_i}^G \trivrep \to \Ind_{\ovrP}^G \ind_{M_i}^M \trivrep$ from \cref{thm-comparison} sends the element $1_{K_i} \in \ind_{K_i}^G \trivrep$ to the element in $\Ind_{\ovrP}^G \ind_{M_i}^M \trivrep$ defined by
    \[
        g \longmapsto
        \begin{cases} m \cdot 1_{M_i} & \text{if $g = \ovrn m n \in \ovrN M N_0 = \ovrP N_0$} \\ 0 & \text{if $g \notin \ovrP N_0$} \end{cases} \, . \qedhere
    \]
\end{proof}

\medskip

Now assume that $R$ is Noetherian.
Then we have Emerton's ordinary parts functor
\[
    \Ord_P^G \colon \Mod_R^{\sm}(G) \longrightarrow \Mod_R^{\sm, \lf{A_M}}(M) \, , \qquad V \mapsto \roundbr[\big]{\Ind_{M^+}^M V^{N_0}}^{\lf{A_M}} \, ,
\]
see \parencite{em1} and \parencite[]{vig}.
When $R$ is Artinian, then on locally admissible representations $\Ord_P^G$ has the following alternative description.

\begin{lemma} \label{lem:comparing-ind-big-small}
    Suppose that $R$ is an Artinian local ring with residue field of characteristic $p$.
    Then, for a locally finite smooth $A_M^+$-representation, the natural $A_M$-equivariant map
    \[
        \roundbr[\big]{\Ind_{A_M^+}^{A_M} V}^{\lf{A_M}} \longrightarrow \ind_{A_M^+}^{A_M} V
    \]
    is an isomorphism.

    Consequently, for a locally admissible smooth $G$-representation $V$, we have a natural isomorphism of smooth $M$-representations
    \[
        \Ord_P^G V \, \cong \, \ind_{M^+}^M V^{N_0} \, .
    \]
\end{lemma}

\begin{proof}
    The first part is \parencite[Lemma 3.2.1]{em2}.
    The second part then follows because for a locally admissible $G$-representation $V$ we have that $V^{N_0}$ is $A_M^+$-locally finite.
\end{proof}

From \cref{cor-describing-right-adjoint} we recover the following result by Vignéras.

\begin{prop}[{\parencite[Corollary 7.3]{vig}}] \label{cor-describing-ord}
    We have
    \[
        \Ord_P^G \, \cong\, (\blank)^{\lf{A_M}} \circ \R_{\ovrP}^G \, ,
    \]
    i.e.\ $\Ord_P^G$ is right adjoint to the restriction of $\Ind_{\ovrP}^G$ to the subcategory $\Mod_R^{\sm, \lf{A_M}}(M) \subseteq \Mod_R^{\sm}(M)$.
\end{prop}

\begin{proof}
    Let $V$ be a smooth $G$-representation.
    Then we have to show that the natural map
    \[
        \colim_i \Ind_{A_M^+}^{A_M} V^{K_i} \longrightarrow \Ind_{A_M^+}^{A_M} V^{N_0} \cong \Ind_{M^+}^M V^{N_0}
    \]
    becomes an isomorphism after applying $(\blank)^{\lf{A_M}}$.
    This follows because both
    \[
        (\blank)^{\lf{A_M}} \colon \Mod_R^{\sm}(A_M) \longrightarrow \Mod_R^{\sm, \lf{}}(A_M)
        \quad \text{and} \quad
        \roundbr[\big]{\Ind_{A_M^+}^{A_M} (\blank)}^{\lf{A_M}} \colon \Mod_R^{\sm}(A_M^+) \longrightarrow \Mod_R^{\sm, \lf{}}(A_M)
    \]
    commute with filtered colimits, see \parencite[Lemma 3.2.2]{em1} for the second assertion.
\end{proof}

Using this we can now deduce that $\R_{\ovrP}^G$ and $\Ord_P^G$ agree on admissible representations, see \cref{cor-admissible-rpg-ord} below.

\begin{lemma} \label{lem:ind-finite}
    Let $V$ be a finite smooth $A_M^+$-representation.
    Then $\Ind_{A_M^+}^{A_M} V$ is again finite.
\end{lemma}

\begin{proof}
    As $V$ is a finitely generated module over a Noetherian ring it has bounded $b_0^{\infty}$-torsion and thus the evaluation map
    \[
        \Ind_{A_M^+}^{A_M} V \, \cong \, \lim_{b_0} V^{K_i} \longrightarrow V
    \]
    is injective.
    Thus $\Ind_{A_M^+}^{A_M} V$ is again finite as desired.
    See also \parencite[Lemma 3.1.5]{em1}.
\end{proof}

\begin{prop} \label{cor-admissible-rpg-ord}
    Let $V$ be an admissible smooth $G$-representation.
    Then we have
    \[
        \R_{\ovrP}^G V \, \cong \, \Ord_P^G V \, .
    \]
\end{prop}

\begin{proof}
    As $V^{K_i}$ is finitely generated over $R$ by assumption the same is true for $\Ind_{A_M^+}^{A_M} V^{K_i}$ by \cref{lem:ind-finite}.
    Thus $\R_{\ovrP}^G V$ is $A_M$-locally finite so that the claim follows from \cref{cor-describing-ord}.
\end{proof}

\medskip
\section{\texorpdfstring{The Right Derived Functors of $\R_{\ovrP}^G$ and $\Ord_P^G$}{The Right Derived Functors of RPbarG and OrdPG}} \label{sec-derived-functors}

From now on we assume that $R$ is an Artinian local ring with residue field of characteristic $p$.

\medskip

\subsection{The Right Derived Functors of $\R_{\ovrP}^G$} \label{subsec-right-derived-r}

For the moment fix a torsionfree compact open subgroup $K \subseteq G$ that admits an Iwahori decomposition with respect to $(N, M, \ovrN)$ and such that we have $a (K \cap N) a^{-1} \subseteq K \cap N$ and $a^{-1} (K \cap \ovrN) a \subseteq K \cap \ovrN$ for all $a \in A_M^+$, and write $\sovrA_M \coloneqq A_M/(K \cap A_M)$ and $\sovrA_M^+ \coloneqq A_M^+/(K \cap A_M)$.

We now first consider the functor
\[
    \R_{\ovrP, K}^G \colon \Mod_R^{\sm}(G) \longrightarrow \Mod_R(\sovrA_M) \subseteq \Mod_R^{\sm}(A_M) \, , \qquad V \mapsto \Ind_{\sovrA_M^+}^{\sovrA_M} V^K \, .
\]
For a smooth $G$-representation $V$ it follows from \cref{cor-describing-right-adjoint} that we have an $A_M$-equivariant isomorphism $\R_{\ovrP}^G V \cong \colim_i \R_{\ovrP, K_i}^G V$.
In the following we write $\bbR^G_{\ovrP} \coloneqq \bbR \R^G_{\ovrP}$ and $\bbR^G_{\ovrP, K} \coloneqq \bbR \R^G_{\ovrP, K}$.

\begin{lemma} \label{lem-derived-comp-invariants-ind}
    For a K-injective complex $I^{\bullet} \in \cK(\Mod_R^{\sm}(G))$ the complex $I^{\bullet, K} \in \cK(\Mod_R(\sovrA_M^+))$ is acyclic for the functor $\Ind_{\sovrA_M^+}^{\sovrA_M} \colon \Mod_R(\sovrA_M^+) \to \Mod_R(\sovrA_M)$.

    Consequently, we have a natural equivalence
    \[
        \bbR_{\ovrP, K}^G \, \simeq \, \bbR {\Ind_{\sovrA_M^+}^{\sovrA_M}} \circ \bbR (\blank)^K
    \]
    of functors $\cD(G) \to \cD(\sovrA_M)$.
\end{lemma}

\begin{proof}
    Choose a projective resolution $P^{\bullet}$ of $R[\sovrA_M]$ as an $R[\sovrA_M^+]$-module.
    Then we have
    \begin{align*}
        \bbR \Ind_{\sovrA_M^+}^{\sovrA_M} I^{\bullet, K}
        & \, \simeq \, \Hom_{\sovrA_M^+}^{\bullet} \roundbr[\big]{P^{\bullet}, \Hom_G^{\bullet}(\ind_K^G \trivrep, I^{\bullet})} \\
        & \, \simeq \, \Hom_G^{\bullet} \roundbr[\big]{\ind_K^G \trivrep \otimes_{R[\sovrA_M^+]} P^{\bullet}, I^{\bullet}} \\
        & \, \simeq \, \Hom_G^{\bullet} \roundbr[\big]{\ind_{\sovrA_M^+}^{\sovrA_M} \ind_K^G \trivrep, I^{\bullet}} \, \simeq \, \Ind_{\sovrA_M^+}^{\sovrA_M} I^{\bullet, K} \, ,
    \end{align*}
    where we use that $R[\sovrA_M]$ is flat over $R[\sovrA_M^+]$ in the third step.
\end{proof}

\begin{lemma} \label{lem:locally-finite-cohomology}
    Let $V^{\bullet}$ be a complex of smooth $G$-representations.

    If $V^{\bullet}$ has admissible cohomology (i.e.\ if the $\HH^n(V^{\bullet})$ are admissible smooth $G$-representations) then $\HH^n(K, V^{\bullet})$ is finite and $\HH^n(N_0, V^{\bullet})$ is $A_M^+$-locally finite.
    If $V^{\bullet}$ has locally admissible cohomology then $\HH^n(K, V^{\bullet})$ and $\HH^n(N_0, V^{\bullet})$ are $A_M^+$-locally finite.
\end{lemma}

\begin{proof}
    As $K$ and $N_0$ have finite $R$-cohomological dimension, see \parencite[Corollaire 1]{ser}, we are reduced to treating the case when $V^{\bullet} = V$ is concentrated in degree $0$.

    For the first part suppose that $V$ is admissible.
    Then $\HH^n(K, V)$ is finite by \parencite[Lemma 3.4.4]{em2} and consequently $\HH^n(N_0, V) = \colim_i \HH^n(K_i, V)$ is $A_M^+$-locally finite.
    Note that we can still apply \parencite[Lemma 3.4.4]{em2} even though the residue field of $R$ may not be finite because using using \parencite[Proposition 2.3]{bru} one sees that \parencite[Proposition 2.1.9]{em2} also holds in the more general setting.

    The second part then follows because profinite group cohomology commutes with filtered colimits; this can be seen by noting that it is computed by a complex of continuous cochains, see \parencite[Section 2.2]{em2}.
\end{proof}

\begin{lemma} \label{lem-finite-acyclic}
    The finite $\sovrA_M^+$-representations are acyclic for $\Ind_{\sovrA_M^+}^{\sovrA_M} \colon \Mod_R(\sovrA_M^+) \to \Mod_R(\sovrA_M)$.
\end{lemma}

\begin{proof}
    Let $V$ be a finite $\sovrA_M^+$-representation.
    Then $V$ is an Artinian $R$-module and thus multiplication by $b_0$ on $V$ satisfies the Mittag-Leffler condition so that we obtain $\bbR^q \Ind_{\sovrA_M^+}^{\sovrA_M} V \, \cong \, \bbR^q \lim_{b_0} V \, = \, 0$.
\end{proof}

\begin{theorem} \label{thm-right-derived-rpg}
    Let $V^{\bullet}$ be a complex of smooth $G$-representations with admissible cohomology.
    Then we have natural isomorphisms of smooth $M$-representations
    \[
        \bbR_{\ovrP}^{G, n} V^{\bullet}
        \, \cong \,
        \colim_i \Ind_{A_M^+}^{A_M} \HH^n(K_i, V^{\bullet})
        \, \cong \,
        \roundbr[\big]{\Ind_{M^+}^M \HH^n(N_0, V^{\bullet})}^{\lf{A_M}}
        \, \cong \,
        \ind_{M^+}^M \HH^n(N_0, V^{\bullet}) \, .
    \]
\end{theorem}

\begin{proof}
    By \cref{lem-derived-comp-invariants-ind} we have a spectral sequence
    \[
        E_2^{p, q} \, = \, \bbR^p \Ind_{\sovrA_M^+}^{\sovrA_M} \HH^q(K, V^{\bullet}) \, \Longrightarrow \, \bbR_{\ovrP, K}^{G, p + q} V^{\bullet} \, .
    \]
    Now the $\HH^q(K, V^{\bullet})$ are finite by \cref{lem:locally-finite-cohomology} and consequently $E_2^{p, q} = 0$ for $p > 0$ by \cref{lem-finite-acyclic}.
    The first of the isomorphisms above is now obtained from setting $K = K_i$ for varying $i$ and taking the colimit.

    Now we use \cref{lem:ind-finite}, the fact that
    \[
        \roundbr[\big]{\Ind_{A_M^+}^{A_M} (\blank)}^{\lf{A_M}} \colon \Mod_R^{\sm}(A_M^+) \longrightarrow \Mod_R^{\sm, \lf{}}(A_M)
    \]
    commutes with filtered colimits, see \parencite[Lemma 3.2.2]{em1}, and \cref{lem:comparing-ind-big-small} to obtain the second and third isomorphisms.
\end{proof}

\begin{cor}\label{cor-right-derived-rpg}
    We have a natural transformation
    \[
        \psi \colon \bbR_{\ovrP}^G \longrightarrow \ind_{M^+}^M \bbR(\blank)^{N_0}
    \]
    of functors $\cD(G) \to \cD(M)$ that is an equivalence on complexes with admissible cohomology.
\end{cor}

\begin{proof}
    Using the description from \cref{cor-describing-right-adjoint} we obtain natural transformations
    \[
        \R_{\ovrP}^G \longrightarrow \Ind_{M^+}^M (\blank)^{N_0} \longrightarrow \ind_{M^+}^M (\blank)^{N_0}
    \]
    of functors $\Mod_R^{\sm}(G) \to \Mod_R^{\sm}(M)$.
    This yields the natural transformations
    \[
        \bbR_{\ovrP}^G \longrightarrow \bbR \roundbr[\big]{\Ind_{M^+}^M (\blank)^{N_0}} \longrightarrow \bbR \roundbr[\big]{\ind_{M^+}^M (\blank)^{N_0}} \, \simeq \, \ind_{M^+}^M \bbR (\blank)^{N_0}
    \]
    of functors $\cD(G) \to \cD(M)$ whose composite we denote by $\psi$.

    We now have a commutative diagram
    \[
    \begin{tikzcd}
        \bbR_{\ovrP}^{G, n} \ar[r] \ar[d]
        & \bbR^n \roundbr[\big]{\Ind_{M^+}^M (\blank)^{N_0}} \ar[r] \ar[d]
        & \bbR^n \roundbr[\big]{\ind_{M^+}^M (\blank)^{N_0}} \ar[d, "\cong"]
        \\
        \colim_i \Ind_{A_M^+}^{A_M} \HH^n(K_i, \blank) \ar[r]
        & \Ind_{M^+}^M \HH^n(N_0, \blank) \ar[r]
        & \ind_{M^+}^M \HH^n(N_0, \blank)
    \end{tikzcd}
    \]
    of functors $\cD(G) \to \Mod_R^{\sm}(M)$ and we have to show that the composition of the upper two horizontal natural transformation is an isomorphism on complexes with admissible cohomology.
    But on such complexes the left vertical natural transformation as well as the composition of the lower two horizontal natural transformations are isomorphisms by \cref{thm-right-derived-rpg}.
\end{proof}

\medskip

\subsection{The Right Derived Functors of $\Ord_P^G$}


We again temporarily fix a compact open subgroup $K \subseteq G$ as in \cref{subsec-right-derived-r} and consider the functor
\[
    \Ord_{P, K}^G \colon \Mod_R^{\sm}(G) \longrightarrow \Mod_R^{\lf{}}(\sovrA_M) \, , \qquad V \mapsto \roundbr[\big]{\R_{\ovrP, K}^G V}^{\lf{\sovrA_M}} = \roundbr[\big]{\Ind_{\sovrA_M^+}^{\sovrA_M} V^K}^{\lf{\sovrA_M}} \, .
\]
For a smooth $G$-representation $V$ we have an $A_M$-equivariant isomorphism $\Ord_P^G V \cong \colim_i \Ord_{P, K_i}^G V$, similarly as for $\R_{\ovrP}^G$.

\begin{lemma} \label{lem-derived-comp-invariants-ind-lf}
    The functors $\Ind_{\sovrA_M^+}^{\sovrA_M} \colon \Mod_R(\sovrA_M^+) \to \Mod_R(\sovrA_M)$ and $\R_{\ovrP, K}^G \colon \Mod_R^{\sm}(G) \to \Mod_R(\sovrA_M)$ both preserve K-injective complexes.

    Consequently, we have natural equivalences
    \begin{align*}
        \bbR \Ord_{P, K}^G
        & \, \simeq \, \bbR (\blank)^{\lf{\sovrA_M}} \circ \bbR \R_{\ovrP, K}^G \\
        & \, \simeq \, \bbR (\blank)^{\lf{\sovrA_M}} \circ \bbR {\Ind_{\sovrA_M^+}^{\sovrA_M}} \circ \bbR (\blank)^K \\
        & \, \simeq \, \bbR \roundbr[\big]{\Ind_{\sovrA_M^+}^{\sovrA_M} (\blank)}^{\lf{\sovrA_M}} \circ \bbR (\blank)^K
    \end{align*}
    of functors $\cD(G) \to \cD(\sovrA_M)$.
\end{lemma}

\begin{proof}
    The functor $\Ind_{\sovrA_M^+}^{\sovrA_M}$ is right adjoint to the exact functor $\Res_{\sovrA_M^+}^{\sovrA_M} \colon \Mod_R(\sovrA_M) \to \Mod_R(\sovrA_M^+)$.
    Observing that $\R_{\ovrP, K}^G \cong \Hom_{R[G]}(\ind_{\sovrA_M^+}^{\sovrA_M} \ind_K^G\trivrep, \blank)$ and using \cref{thm-comparison}, we see that $\R_{\ovrP, K}^G$ is right adjoint to the functor
    \[
        \Mod_R(\sovrA_M) \longrightarrow \Mod_R^{\sm}(G) \, , \qquad V \mapsto \roundbr[\big]{\Ind_{\ovrP}^G \ind_{K \cap M}^M \trivrep} \otimes_{R[\sovrA_M]} V \, .
    \]
    As $\ind_{K \cap M}^M \trivrep$ is free over $R[\sovrA_M]$, also this left adjoint is exact.
    
    From these observations it follows that both right adjoints $\Ind_{\sovrA_M^+}^{\sovrA_M}$ and $\R_{\ovrP, K}^G$ preserve K-injective complexes.
    For the second part of the lemma we also use \cref{lem-derived-comp-invariants-ind}.
\end{proof}

\begin{lemma} \label{lem-locally-finite-acyclic}
        Let $S$ be a finitely generated commutative $R$-algebra and write $\Mod_R^{\lf{}}(S)$ for the full subcategory of $\Mod_R(S) = \Mod_S$ spanned by those $S$-modules $V$ such that $S \cdot v \subseteq V$ is finitely generated over $R$ for all $v \in V$.
        This subcategory contains $0$ and is stable under taking subobjects, quotients, extensions and arbitrary direct sums.
        Moreover, the inclusion functor $\Mod_R^{\lf{}}(S) \to \Mod_R(S)$ preserves injective objects.   
\end{lemma}

\begin{proof}
    The first part of the claim is a straightforward computation.
    In order to prove the second part, it suffices to see that for an injective $S$-module $I$ the submodule of locally finite vectors $I^{\lf{}} \subseteq I$ is again an injective $S$-module.
    For this we note that
    \[
        I^{\lf{}} \, = \, \colim_{\ka \in \cA^{\op}} I[\ka] \, ,
    \]
    where $\cA$ is the (cofiltered) partially ordered set of those ideals $\ka \subseteq S$ such that $S/\ka$ is finitely generated over $R$ and $I[\ka] \subseteq I$ denotes the $\ka$-torsion submodule.
    Now, as for every $\ka \in \cA$ also every power $\ka^n$ is again contained in $\cA$, the proof of \parencite[Lemma III.3.2]{har} shows that $I^{\lf{}}$ is injective as desired.    
\end{proof}

\begin{cor}\label{cor-acyclic}
    The locally finite $\sovrA_M^+$-representations are acyclic for
        \[
            \roundbr[\big]{\Ind_{\sovrA_M^+}^{\sovrA_M} (\blank)}^{\lf{\sovrA_M}} \colon \Mod_R(\sovrA_M^+) \longrightarrow \Mod_R^{\lf{}}(\sovrA_M) \, .
        \]
\end{cor}

\begin{proof}
    \cref{lem-locally-finite-acyclic} applies to the inclusion functor $\Mod_R^{\lf{}}(\sovrA_M^+) \to \Mod_R(\sovrA_M^+)$.
    Thus it suffices to show that $(\Ind_{\sovrA_M^+}^{\sovrA_M} (\blank))^{\lf{\sovrA_M}}$ is exact when restricted to $\Mod_R^{\lf{}}(\sovrA_M^+)$.
    But on this category it agrees with the exact functor $\ind_{\sovrA_M^+}^{\sovrA_M}$ by \cref{lem:comparing-ind-big-small}.
\end{proof}

\begin{theorem} \label{thm-right-derived-ord}
    Let $V^{\bullet}$ be a complex of smooth $G$-representations with locally admissible cohomology.
    Then we have natural isomorphisms of smooth $M$-representations
    \[
        \bbR^n \Ord_P^G V^{\bullet}
        \, \cong \,
        \roundbr[\big]{\Ind_{M^+}^M \HH^n(N_0, V^{\bullet})}^{\lf{A_M}}
        \, \cong \,
        \ind_{M^+}^M \HH^n(N_0, V^{\bullet}) \, .
    \]
\end{theorem}

\begin{proof}
    The structure of this proof is almost identical to the one of \cref{thm-right-derived-rpg}.
    By \cref{lem-derived-comp-invariants-ind-lf} we have a spectral sequence
    \[
        E_2^{p, q} \, = \, \bbR^p \roundbr[\big]{\Ind_{\sovrA_M^+}^{\sovrA_M} (\blank)}^{\lf{\sovrA_M}} \HH^q(K, V^{\bullet}) \Longrightarrow \bbR^{p + q} \Ord_{P, K}^G \, .
    \]
    Now the $\HH^q(K, V^{\bullet})$ are $\sovrA_M^+$-locally finite by \cref{lem:locally-finite-cohomology} and consequently $E_2^{p, q} = 0$ for $p > 0$ by \cref{cor-acyclic}.
    The first of the isomorphisms above is now obtained from setting $K = K_i$ for varying $i$, taking the colimit, and using once more that $(\Ind_{A_M^+}^{A_M}(\blank))^{\lf{A_M}}$ commutes with filtered colimits.

    The second isomorphism then is again obtained from \cref{lem:comparing-ind-big-small}.
\end{proof}

\begin{cor} \label{cor-ord-rpg}
    The composition
    \[
        \bbR \Ord_P^G \longrightarrow \bbR_{\ovrP}^G \xlongrightarrow{\psi} \ind_{M^+}^M \bbR (\blank)^{N_0}
    \]
    of natural transformations of functors $\cD(G) \to \cD(M)$ is an equivalence on complexes with locally admissible cohomology.

    Consequently, $\bbR \Ord_P^G \to \bbR_{\ovrP}^G$ is an equivalence on complexes with admissible cohomology.
\end{cor}

\begin{proof}
    The first part follows from \cref{thm-right-derived-ord} in the same way as \cref{cor-right-derived-rpg} follows from \cref{thm-right-derived-rpg}.
    The second part then follows from \cref{cor-right-derived-rpg}.
\end{proof}

\medskip

\appendix

\section{Second Adjointness (joint with Claudius Heyer)} \label{sec-adj}

We continue to assume that $R$ is an Artinian local ring with residue field of characteristic $p$.
Following \parencite{heyer-left-adjoint} and \parencite{heyer-geom} we write $\bbL_N \colon \cD(P) \to \cD(M)$ for the left adjoint to inflation and $\bbL_P^G \colon \cD(G) \to \cD(M)$ for the composition of $\bbL_N$ with restriction; $\bbL_P^G$ is then left adjoint to the parabolic induction functor $\Ind_P^G$.
Note that in \parencite{heyer-left-adjoint}, the existence of these left adjoints is only shown in the case when $R$ is a field, but the arguments generalize; for a different proof of the existence for general $R$, see \parencite[Lemma 5.4.6]{heyer-mann}.

The Geometrical Lemma \parencite[Theorem 3.3.5]{heyer-geom} gives a natural isomorphism
\[
	\omega \otimes_R^{\bbL} (\blank)
	\, \simeq \,
	\bbL_N \cInd_{\ovrP}^{\ovrP P}
\]
of functors $\cD(M) \to \cD(M)$, where $\omega \coloneqq \bbL_N \cInd_{\ovrP}^{\ovrP P} \trivrep$ is a character sitting in cohomological degree $-\dim_{\Qp} N$, see \cref{lem:hecke-description-left-adjoint} below.
Composing this with $\bbL_N \cInd_{\ovrP}^{\ovrP P} \to \bbL_P^G \Ind_{\ovrP}^G$ and passing to the right mate yields a natural transformation
\[
	\alpha \colon \omega \otimes_R^{\bbL} \bbR_{\ovrP}^G \longrightarrow \bbL_P^G
\]
of functors $\cD(M) \to \cD(M)$.
In \cref{thm-sec-adj} below we prove that $\alpha$ is an equivalence on complexes with admissible cohomology.

\medskip

We start with some preliminary considerations and recollections. 

\begin{lemma} \label{lem-proj-formula}
    The functor $\ind_{M^+}^M \colon \cD(M^+) \to \cD(M)$ is symmetric monoidal, i.e.\ for complexes $V^{\bullet}$ and $W^{\bullet}$ of smooth $M^+$-representations we have a natural equivalence
	\[
		\ind_{M^+}^M \roundbr[\big]{V^{\bullet} \otimes_R^{\bbL} W^{\bullet}}
		\, \simeq \,
		\ind_{M^+}^M V^{\bullet} \otimes_R^{\bbL} \ind_{M^+}^M W^{\bullet}
	\]
	in $\cD(M)$.

	Consequently, for a complex $V^{\bullet}$ of smooth $M^+$-representations and a complex $W^{\bullet}$ of smooth $M$-representations $W^{\bullet}$, we obtain the projection formula
	\[
		\ind_{M^+}^M \roundbr[\big]{V^{\bullet} \otimes_R^{\bbL} W^{\bullet}}
		\, \simeq \,
		\ind_{M^+}^M V^{\bullet} \otimes_R^{\bbL} W^{\bullet}
	\]
	in $\cD(M)$.
\end{lemma}

\begin{proof}
	As the restriction functor $\Res_{M^+}^M \colon \cD(M) \to \cD(M^+)$ carries a natural symmetric monoidal structure we obtain an oplax symmetric monoidal structure on its left adjoint $\ind_{M^+}^M$ and our task is to show that this structure is in fact symmetric monoidal.
	To this end, we observe that we have a $2$-commutative diagram
	\[
	\begin{tikzcd}
		\cD(M^+) \ar[r, "{\ind_{M^+}^M}"] \ar[d]
		& \cD(M) \ar[d]
		\\
		\cD(\Mod_R(M^+)) \ar[r, "{\ind_{M^+}^M}"]
		& \cD(\Mod_R(M))
	\end{tikzcd}
	\]
	where the vertical functors are conservative and symmetric monoidal and the horizontal functors are oplax symmetric monoidal (for the lower functor this follows from the same argument as for the upper one).

	It thus suffices to show the symmetric monoidality of the lower functor, i.e.\ we need to show that the natural morphism
	\[
		\ind_{M^+}^M \roundbr[\big]{V^{\bullet} \otimes_R^{\bbL} W^{\bullet}} \longrightarrow \ind_{M^+}^M V^{\bullet} \otimes_R^{\bbL} \ind_{M^+}^M W^{\bullet}
	\]
	is an equivalence for all $V^{\bullet}, W^{\bullet} \in \cD(\Mod_R(M^+))$.
	As both sides are exact and commute with direct sums in $V^{\bullet}$ and $W^{\bullet}$ we can reduce to the case where $V^{\bullet} = W^{\bullet} = R[M^+]$ as this object is a generator of $\cD(\Mod_R(M^+))$.
	Writing $R[M] = \colim_{b_0} R[M^+]$, the morphism above identifies with the morphism
	\[
		\colim_{b_0 \otimes b_0} R[M^+] \otimes_R R[M^+] \longrightarrow \colim_{(b_0 \otimes 1, 1 \otimes b_0)} R[M^+] \otimes_R R[M^+]
	\]
	in $\Mod_R(M)$, where the colimit on the right side is taken over the index set $\ZZ_{\geq 0}^2$; it is an isomorphism by the cofinality of the diagonal $\ZZ_{\geq 0} \to \ZZ_{\geq 0}^2$.

	The second part now follows from the first part using the fully faithfulness of $\Res_{M^+}^M$.
\end{proof}

\medskip

Let us now observe that for a smooth $M$-representation $W$ the $M^+$-equivariant map $W \to (\Ind_{\ovrP}^G W)^{N_0}$ from \cref{cor-describing-right-adjoint} factors through $(\cInd_{\ovrP}^{\ovrP P} W)^{N_0}$; this follows immediately from the given explicit description. 
In this way we obtain a natural transformation
\[\label{equ-hash}\tag{$\#$}
	\theta \colon \id_{\cD(M)} \longrightarrow \bbR \roundbr[\big]{\ind_{M^+}^M (\blank)^{N_0} \cInd_{\ovrP}^{\ovrP P}}
	\, \simeq \,
	\ind_{M^+}^M \bbR (\blank)^{N_0} \cInd_{\ovrP}^{\ovrP P}
\]
of functors $\cD(M) \to \cD(M)$. Here we also used that $\cInd_{\ovrP}^{\ovrP P}$ and $\ind_{M^+}^M$ are exact and that
\[
	\cInd_{\ovrP}^{\ovrP P} W^{\bullet}
	\, \cong \,
	\cInd_1^N W^{\bullet} \, \cong \,
	\bigoplus_{N/N_0} \Ind_1^{N_0} W^{\bullet}
\]
is acyclic for $(\blank)^{N_0}$ for every complex $W^{\bullet} \in \cK(\Mod_R^{\sm}(M))$.
By construction we have a commutative diagram
\[
\begin{tikzcd}
	\id_{\cD(M)} \ar[r, "\eta"] \ar[d, "\theta"]
	& \bbR_{\ovrP}^G \Ind_{\ovrP}^G \ar[d, "\psi_{\Ind_{\ovrP}^G}"]
	\\
	\ind_{M^+}^M \bbR (\blank)^{N_0} \cInd_{\ovrP}^{\ovrP P} \ar[r]
	& \ind_{M^+}^M \bbR (\blank)^{N_0} \Ind_{\ovrP}^G
\end{tikzcd}
\]
of functors $\cD(M) \to \cD(M)$, where $\eta$ is the unit of the adjunction between $\Ind_{\ovrP}^G$ and $\bbR_{\ovrP}^G$ and $\psi$ was introduced in \cref{cor-right-derived-rpg}.
For the following lemma compare also \parencite[Proposition 4.2.7]{em1}.

\begin{lemma}\label{lem-theta-equivalence}
	The natural transformation $\theta$ from \eqref{equ-hash} is an equivalence.
\end{lemma}

\begin{proof}
	As $\theta$ is a natural transformation between right derived functors it suffices to show that the corresponding natural transformation
	\[
		\theta' \colon \id_{\Mod_R^{\sm}(M)} \to \ind_{M^+}^M (\blank)^{N_0} \cInd_{\ovrP}^{\ovrP P}
	\]
	is an isomorphism.
	For a smooth $M$-representation $W$ we now have
	\begin{align*}
		\ind_{M^+}^M (\cInd_{\ovrP}^{\ovrP P} W)^{N_0}
		& \, \cong \, \ind_{M^+}^M \roundbr[\big]{\cInd_{\ovrP }^{\ovrP P} \trivrep \otimes_R W}^{N_0} \\
		& \, \cong \, \ind_{M^+}^M\roundbr[\big]{(\cInd_{\ovrP }^{\ovrP P}\trivrep)^{N_0} \otimes_R W}\\
		& \, \cong \, \ind_{M^+}^M(\cInd_{\ovrP }^{\ovrP P}\trivrep)^{N_0} \otimes_R W \, ,
	\end{align*}
	where we used the projection formulas from \parencite[Lemma~2.2.3]{heyer-geom}, the proof of \parencite[Proposition~3.4.18]{heyer-left-adjoint} and \cref{lem-proj-formula}.
	Under these isomorphisms $\theta'_W$ identifies with $\theta'_{\trivrep} \otimes \id_W$ so that it in fact suffices to show that $\theta'_{\trivrep}$ is an isomorphism.

	We view the elements of $\cInd_{\ovrP}^{\ovrP P}\trivrep$ as compactly supported locally constant functions $f \colon N \to R$; the action of $M$ is then induced by the conjugation action on $N$.
	Consequently, the $M^+$-action on $(\cInd_{\ovrP}^{\ovrP P} \trivrep)^{N_0}$ is given concretely by
	\[
		(m \bullet f)(u) \, = \, \sum_{n \in N_0/m N_0 m^{-1}} f(m^{-1} u n m)
	\]
	for $m \in M^+$ and $u \in N$.

	For $n \in N/N_0$ let us write $1_{n N_0} \in (\cInd_{\ovrP}^{\ovrP P} \trivrep)^{N_0}$ for the characteristic function on $n N_0$. These elements form an $R$-basis of $(\cInd_{\ovrP}^{\ovrP P} \trivrep)^{N_0}$.
	We note that $m \bullet 1_{n N_0} = 1_{m n m^{-1} N_0}$ for all $m \in M^+$ and $n \in N$ and that $\theta'_{\trivrep}(1) = 1 \otimes 1_{N_0}$.
	The map $(\cInd_{\ovrP }^{\ovrP P}\trivrep)^{N_0} \to \trivrep$ given by $1_{n N_0} \mapsto 1$ is evidently $M^+$-equivariant and hence induces an $M$-equivariant map
	\begin{align*}
		\rho \colon \ind_{M^+}^M (\cInd_{\ovrP }^{\ovrP P}\trivrep)^{N_0} \longrightarrow \trivrep
	\end{align*}
	that we claim is inverse to $\theta'_{\trivrep}$.
	Clearly, $\rho \circ \theta'_{\trivrep} = \id_{\trivrep}$.
	Conversely, the identity $\theta_{\trivrep} \circ \rho = \id$ can be checked on each $1_{nN_0}$, which amounts to checking that the images of $1_{nN_0}$ and $1_{N_0}$ in $\ind_{M^+}^M (\cInd_{\ovrP}^{\ovrP P} \trivrep)^{N_0}$ agree.
	Now, for any $n \in N$ we may choose $m \in M^+$ such that $m n m^{-1} \in N_0$.
	Then $m \bullet 1_{n N_0} = 1_{N_0} = m \bullet 1_{N_0}$, which implies the claim.
\end{proof}

Recall from \parencite[Lemma 3.4.12]{heyer-left-adjoint} that the functor $\bbR (\blank)^{N_0} \colon \cD(P^+) \longrightarrow \cD(M^+)$ admits a right adjoint $F_{P^+}^{M^+}$.
We set $\omega^+ \coloneqq F_{P^+}^{M^+}(\trivrep)$.

\begin{lemma} \label{lem:hecke-description-left-adjoint}
	The object $\omega^+ \in \cD(P^+)$ is an invertible character that is trivial on $N_0$ sitting in cohomological degree $-d \coloneqq -\dim_{\Qp} N$ and we have an equivalence $\omega \simeq \omega^+$ in $\cD(P^+)$.

	Moreover, we have a natural equivalence
	\[
		\bbL_N
		\, \simeq \,
		\omega \otimes_R^{\bbL} \ind_{M^+}^M \bbR (\blank)^{N_0}
	\]
	of functors $\cD(P) \to \cD(M)$.
\end{lemma}

\begin{proof}
	We start by noting that we have an equivalence $\omega^+ \simeq F_{P_0}^{M_0}(\trivrep)$ in $\cD(P_0)$ where $F_{P_0}^{M_0}$ is the right adjoint of $\bbR (\blank)^{N_0} \colon \cD(P_0) \to \cD(M_0)$, see \parencite[Lemma 3.4.13]{heyer-left-adjoint}.
	If $R$ is a field, then this is a character that is trivial on $N_0$ sitting in degree $-d$ by \parencite[Proposition 3.1.10]{heyer-left-adjoint}.
	For general $R$ the same is true as can be seen as follows, using the results of \parencite{heyer-mann}.

	In the notation of \textit{op.\ cit.} the functor $F^{M_0}_{P_0}$ corresponds to $f^!$, where $f \colon \ast/P_0 \to \ast/M_0$ is the map on classifying stacks induced by the projection $P_0 \to M_0$.
	Consider the base-change diagram
	\[
	\begin{tikzcd}
		\ast/N_0 \ar[d,"g'"'] \ar[r,"f'"]
		& \ast \ar[d,"g"]
		\\
		\ast/P_0 \ar[r,"f"']
		& \ast/M_0
	\end{tikzcd} \, .
	\]
	As the right vertical map is an $R$-smooth cover and $f'$ is $R$-smooth by \cite[Example~5.3.22]{heyer-mann}, we conclude from \cite[Lemma~4.5.7]{heyer-mann} that $f$ is $R$-smooth.
	By suave base-change, see \cite[Lemma~4.5.13]{heyer-mann}, it suffices to see that $g'^*f^!(\trivrep) \simeq f'^!(\trivrep) \simeq F^1_{N_0}(\trivrep)$ is the trivial character of $N_0$ sitting in degree $-d$.
	But this follows from the explicit description in \cite[Example~5.3.22]{heyer-mann}:
	We have $F^1_{N_0}(\trivrep) \simeq \delta_{N_0}[d]$, where $\delta_{N_0}$ is given as the composition
	\[
		N_0 \xlongrightarrow{\det(\kn) \cdot \abs{\det(\kn)}_p} \ZZ_p^{\times} \to R^{\times} \, .
	\]
	Here, $\kn \coloneqq \Lie_{\Qp}(N)$ and $\det(\kn) \colon N \to \Qp^{\times}$ denotes the determinant of the adjoint representation of $N$ on $\kn$.
	As $\bfN$ is a unipotent algebraic group, the character $\det(\kn)$ is trivial, and hence so is $\delta_{N_0}$.

	At this point we have shown that $\omega^+ \simeq \delta^+[d]$ for some character $\delta^+$ of $M^+$.
	From the description \parencite[Corollary 3.4.20]{heyer-left-adjoint} of the left adjoint $\bbL_{N_0}$ of the inflation functor $\cD(M^+) \to \cD(P^+)$ and the projection formula for $\bbR (\blank)^{N_0}$, we now deduce that
	\[
		\trivrep
		\, \cong \,
		\HH^0 \roundbr[\big]{\bbL_{N_0} \trivrep}
		\, \cong \,
		\HH^0 \roundbr[\big]{\bbR (\blank)^{N_0} (\omega^+ \otimes_R^{\bbL} \trivrep)}
		\, \cong \,
		\HH^0 \roundbr[\big]{\omega^+ \otimes_R^{\bbL} \bbR (\blank)^{N_0}(\trivrep)}
		\, \cong \,
		\delta^+ \otimes_R \HH^d(N_0, \trivrep)
	\]
	as $M^+$-representations.
	It follows that $\delta^+$ is invertible and hence extends to a character of $M$.
	Finally we compute
	\[
		\bbL_N
		\, \simeq \,
		\ind_{M^+}^M \bbR (\blank)^{N_0} \roundbr[\big]{\omega^+\otimes_R^{\bbL} (\blank)}
		\, \simeq \,
		\omega^+ \otimes_R^{\bbL} \ind_{M^+}^M \bbR (\blank)^{N_0} \, ,
	\]
	where we have used the description \parencite[Proposition 3.4.6]{heyer-left-adjoint} and the projection formulas for $\bbR (\blank)^{N_0}$ and $\ind_{M^+}^M$.
	Applying this to $\cInd_{M}^P \trivrep$ and using \cref{lem-theta-equivalence} we deduce $\omega^+ \simeq \omega$ as desired.
\end{proof}

\begin{theorem}[Second Adjointness] \label{thm-sec-adj}
	Let $V^{\bullet}$ be a complex of smooth $G$-representations with admissible cohomology.
	Then the morphism
 	\[
 		\alpha_{V^{\bullet}} \colon \omega \otimes_R^{\bbL} \bbR_{\ovrP}^G V^{\bullet} 
 		\longrightarrow \bbL_P^G V^{\bullet}
 	\]
 	is an equivalence. 
\end{theorem}

\begin{proof}
	Under the identification $\omega \otimes_R^{\bbL} (\blank) \simeq \bbL_N \cInd_{\ovrP}^{\ovrP P}$, the natural transformation $\alpha$ corresponds to $\bbL_N \beta$, where $\beta$ denotes the composition
	\[
		\cInd_{\ovrP}^{\ovrP P} \bbR_{\ovrP}^G \longrightarrow \Ind_{\ovrP}^G \bbR_{\ovrP}^G \xlongrightarrow{\varepsilon} \Res_P^G \, ,
	\]
	the second natural transformation being the counit of the adjunction between $\Ind_{\ovrP}^G$ and $\bbR_{\ovrP}^G$. This follows immediately from the definition of $\alpha$.

	Now consider the commutative diagram
	\[
	\begin{tikzcd}
		\omega \otimes_R^{\bbL} \bbR^G_{\ovrP} \ar[d,"\simeq"', "\theta"] \ar[r,"\eta"] \ar[rr, rounded corners, "\id", to path = {|- ([yshift=2ex]\tikztotarget.north) [pos=0.75]\tikztonodes -- (\tikztotarget)}]
		& \omega \otimes_R^{\bbL} \bbR^G_{\ovrP} \Ind_{\ovrP}^G \bbR^G_{\ovrP} \ar[d, "\psi"] \ar[r,"\varepsilon"]
		& \omega \otimes_R^{\bbL} \bbR^G_{\ovrP} \ar[d, "\psi"]
		\\
		\omega \otimes_R^{\bbL} \ind_{M^+}^M \bbR (\blank)^{N_0} \cInd_{\ovrP}^{\ovrP P} \bbR^G_{\ovrP} \ar[d,"\simeq"'] \ar[r]
		& \omega \otimes_R^{\bbL} \ind_{M^+}^M\bbR (\blank)^{N_0} \Ind_{\ovrP}^G \bbR^G_{\ovrP} \ar[d,"\simeq"'] \ar[r,"\varepsilon"]
		& \omega \otimes_R^{\bbL} \ind_{M^+}^M \bbR (\blank)^{N_0} \ar[d,"\simeq"']
		\\
		\bbL_N \cInd_{\ovrP}^{\ovrP P} \bbR^G_{\ovrP} \ar[r] 
		\ar[rr,rounded corners, "\bbL_N \beta"', to path = {|- ([yshift=-2ex]\tikztotarget.south) [pos=0.75]\tikztonodes -- (\tikztotarget)}]
		& \bbL_P^G \Ind_{\ovrP}^G \bbR^G_{\ovrP} \ar[r,"\varepsilon"]
		& \bbL_P^G
	\end{tikzcd} \, ,
	\]
	where the lower vertical natural transformations are the equivalences given by \cref{lem:hecke-description-left-adjoint} and the upper left vertical natural transformation is an equivalence by \cref{lem-theta-equivalence}.
	On complexes with admissible cohomology the top right vertical natural transformation is an equivalence by \cref{cor-right-derived-rpg}.
	From this it follows that on such complexes also the composition $\bbL_N \beta$ of the lower horizontal natural transformations is an equivalence as desired.
\end{proof}

\medskip

\printbibliography

\end{document}